\documentclass{article}
\pdfoutput=1
\usepackage[margin=1in]{geometry}
\usepackage{amsmath,amsthm,amssymb}
\usepackage{enumerate}
\usepackage{MnSymbol}
\usepackage{graphicx}
\usepackage[justification=centering,format=hang,singlelinecheck=false]{subfig}
\usepackage[font=small,singlelinecheck=off]{caption}
\usepackage{color}

\providecommand{\N}{\ensuremath\mathbb N}
\providecommand{\R}{\ensuremath \mathbb R}
\providecommand{\spt}{\ensuremath \text{spt}}

\newtheorem{thm}{Theorem}
\newtheorem{assum}[thm]{Assumption}
\newtheorem{rem}[thm]{Remark}
\newtheorem{lem}[thm]{Lemma}

\newcommand{\LF}{\mathcal{L}_F}
\newcommand{\Lf}{\mathcal{L}_f}
\newcommand{\Lg}{\mathcal{L}_g}
\newcommand{\M}{\mathcal{M}}

\begin{document}
  \title{Control Synthesis for Nonlinear Optimal Control via Convex Relaxations
  \thanks{This work was supported by Ford Motor Company.}}
  \author{Pengcheng Zhao%
  \thanks{P. Zhao and R. Vasudevan are with the Department of Mechanical Engineering, University of Michigan, Ann Arbor, MI 48109 USA {\tt\scriptsize \{pczhao,ramv\}@umich.edu}}
  \and
  Shankar Mohan%
  \thanks{S. Mohan is with the Department of Electrical Engineering and Computer Science, University of Michigan, Ann Arbor, MI 48109 USA {\tt\scriptsize elemsn@umich.edu}}
  \and 
  Ram Vasudevan%
  \footnotemark[2]
  }
  \maketitle
  \begin{abstract}
      This paper addresses the problem of control synthesis for nonlinear optimal control problems in the presence of state and input constraints. 
      The presented approach relies upon transforming the given problem into an infinite-dimensional linear program over the space of measures. 
      To generate approximations to this infinite-dimensional program, a sequence of Semi-Definite Programs (SDP)s is formulated in the instance of polynomial cost and dynamics with semi-algebraic state and bounded input constraints.
      A method to extract a polynomial control function from each SDP is also given. 
      This paper proves that the controller synthesized from each of these SDPs generates a sequence of values that converge from below to the value of the optimal control of the original optimal control problem.
      In contrast to existing approaches, the presented method does not assume that the optimal control is continuous while still proving that the sequence of approximations is optimal. 
      Moreover, the sequence of controllers that are synthesized  using the presented approach are proven to converge to the true optimal control. 
      The performance of the presented method is demonstrated on three examples.
  \end{abstract}
  \section{Introduction}

A variety of engineering problems require searching for optimal system trajectories while satisfying certain constraints \cite{bertsekas1995dynamic}.
Despite the numerous applications for these optimal control problems, they remain challenging to solve.
Assorted theoretical approaches have been proposed to address these problems including the maximum principle \cite{berkovitz2013optimal} and the Hamilton-Jacobi-Bellman Equation \cite{bardi2008optimal}. 
Due to the challenges associated with applying these methods, various numerical techniques such as direct methods that rely upon gradient descent type algorithms \cite{von1992direct}, and multiple shooting methods that require solving a two-point boundary value problem \cite{polak2012optimization}, have been extensively employed.
Though rendering the optimal control problem more amenable to computation, these approaches struggle to find global minimizers. 

By recasting the optimal control problem into its weak formulation, one can write the same problem as an infinite dimensional Linear Program (LP) over the space of measures \cite{vinter1993}. 
Approximations to this infinite dimensional linear program have recently been developed using Semi-Definite Programming (SDP) hierarchies \cite{lasserre2008}.
This approach has proven an effective method to tractably construct a sequence of lower bounds to the optimal value of the original optimal control problem even in the presence of nonlinear dynamics and states constraints. 

Unfortunately the construction of the optimal controller that achieves this optimal value has remained challenging. 
For example, approaches which rely upon iterating on the extracted solution \cite{henrion2008nonlinear} or assuming the differentiability of the optimal control \cite{korda2016} have been proposed. 
The former method is unable to guarantee the optimality of the generated controller. 
Though the latter approach is able to prove convergence of the controller, a value function, and even a rate of convergence more recently \cite{korda2016}, it assumes differentiability of the optimal control, which is a restrictive assumption \cite{vinter2010optimal}.
This paper addresses these limitations by proposing a method to perform optimal control synthesis for these nonlinear, state-constrained optimal control problems that are solved using SDP approximations to the weak formulation. 

The contributions of this paper are as follows: we formulate the optimal control problem as an infinite-dimensional linear program over the space of non-negative measures, which is amenable to the extraction of the optimal control.
To numerically solve this infinite-dimensional LP, we construct a sequence of relaxations in terms of finite-dimensional SDPs and illustrate how a polynomial control law can be
extracted from a solution to any of the SDPs.
This extraction method extends a recently developed method to perform control synthesis for control problems formulated in the weak sense \cite{majumdar2014}. 
Finally, we prove the convergence of this sequence of control laws to the optimal controller. 

This paper is organized as follows: Section \ref{sec:prelim} defines the problem of interest and key results about occupation measures; Section \ref{sec:LP} formulates the optimal control problem as an infinite-dimensional LP and proves its equivalence to the original optimal control problem; Section \ref{sec:implementation} provides a means to solve the infinite-dimensional problem using a sequence of SDP relaxations, and control extraction method is discussed; Section \ref{sec:examples} demonstrates the performance of our approach on three examples.

  \section{Preliminaries}
\label{sec:prelim}

In this section, we introduce necessary notation for this paper, and formalize our problem of interest.
We make substantial use of measure theory in this paper, and the unfamiliar reader may wish to consult \cite{bogachev2007} for an introduction.

\subsection{Notation}

Given an element $y \in \R^{n}$, let $[y]_i$ denote the $i$-th component of $y$. 
We use the same convention for elements belonging to any multidimensional vector space. 
Let $\R[y]$ denote the ring of real polynomials in the variable $y$. 
Suppose $K$ is a compact Borel subset of $\R^n$, then we let $C(K)$ be the space of continuous functions on $K$ and $\M(K)$ be the space of finite signed Radon measures on $K$, whose positive cone $\M_+(K)$ is the space of unsigned Radon measures on $K$. 
Since any measure $\mu \in \M(K)$ can be viewed as an element of the dual space to $C(K)$, the duality pairing of $\mu$ on a test function $v \in C(K)$ is:
\begin{equation}
\langle \mu, v \rangle := \int_K v(z) \mu(z).
\end{equation}
For any $\mu \in \M(K)$, we let the support of $\mu$ be denoted as $\spt(\mu)$.
A probability measure is a non-negative measure whose integral is one.


\subsection{Problem Formulation}
Consider a control-affine system:
\begin{equation}
    \label{eq:def:F}
    \dot{x}(t) = f(t,x(t)) + g(t,x(t)) u(t) =:F(t,x(t),u(t))
\end{equation}
where $f:[0,\infty) \times \R^n \to \R^n$ is a continuous function, $g[0,\infty) \times \R^n \to \R^{n \times m}$ is a continuous function, $x(t) \in \R^n$ is the system's state at time $t$, and $u(t) \in \R^m$ is a control action at time $t$. 
Furthermore, the control $u$ is a Borel measurable function defined on an interval $[0,T]$ which satisfies the input constraint:
\begin{equation}
\label{eq:def:U}
u(t) \in U = [a_1, b_1] \times \cdots \times [a_m,b_m]
\end{equation}
where $\{a_1, \cdots, a_m,b_1,\cdots, b_m\} \subset \mathbb{R}$. 
\begin{rem}
\label{rem:compact}
Without loss of generality, we assume $U$ is a closed unit box, i.e. $U = [-1,1]^m$ (since $g$ can be arbitrarily shifted and scaled).
\end{rem}

The objective of this paper is to find a control that satisfies this input constraint and an associated finite-time trajectory of the system beginning from a given initial condition that satisfies a set of state constraints while reaching a target set and minimizing a user-specified cost function. 
To formulate this problem, we first define the state constraint and target sets as $X$ and $X_T$, respectively, with $X_T \subset X \subset \R^n$.
We further assume that:
\begin{assum}
$X$ and $X_T$ are compact sets.
\end{assum}
\noindent The compactness ensures our optimization problem is well-posed. 

Next, we define an \emph{admissible trajectory} of this system. Given a point $x_0 \in \R^n$ and a $T > 0$, a control $u :[0,T] \to U$ is said to be \emph{admissible} if there exists an absolutely continuous function $x : [0,T] \to \R^n$ such that:
\begin{enumerate}[(i)]
\item $x(0) = x_0$ and  $x(T) \in X_T$,
\item $x(t) \in X$ for all  $t \in [0,T]$, and
\item $\dot{x}(t) = F(t,x(t),u(t))$ for almost every $t \in [0,T]$.
\end{enumerate}
The function $x$ which satisfies these requirements and corresponds to the admissible control $u$ is called an \emph{admissible trajectory}, and the pair $(x,u)$ is called an \emph{admissible pair}. 
We denote the space of admissible trajectories and controls by ${\cal X}_T$ and ${\cal U}_T$, respectively.
The space of admissible pairs is denoted as ${\cal P}_T \subset {\cal X}_T \times {\cal U}_T$.
Note that the ODE may not admit a unique solution, so the behavior of the system may not be fully characterized by the control, but instead by the admissible pair.
Finally, for each admissible pair $(x,u)$ the cost is defined as:
\begin{equation}
\label{eq:J}
J(x,u) = \int_0^T h(t,x(t),u(t)) \,dt + H(x(T))
\end{equation}
where $h:[0,+\infty) \times \mathbb{R}^n \times \mathbb{R}^m \rightarrow \mathbb{R}$ and $H:\mathbb{R}^n \rightarrow \mathbb{R}$ are Borel measurable functions. 

Our goal is to solve the optimal control problem to find an admissible pair $(x^*,u^*)$ that minimizes the cost in Equation \eqref{eq:J}.
That is, we consider the following optimal control problem:
\begin{flalign}
  &&\inf_{(x,u)} &\phantom{4}\,\int_{0}^T h(t,x(t),u(t))\,dt + H(x(T))  &&(OCP) \nonumber\\
  &&\text{s.t.} &\phantom{4} \dot{x}(t) = F(t,x(t),u(t)) \text{ a.e. } t \in [0,T], \nonumber\\
  &&& \phantom{4} u(t)\in U \quad \forall t\in [0,T], \nonumber\\
  &&& \phantom{4} x(t)\in X \quad \forall t\in [0,T], \nonumber\\
  &&& \phantom{4} x(T) \in X_T, \nonumber\\
  &&& \phantom{4} x(0)=x_0, \nonumber
\end{flalign}
where optimization is over the space of measurable control inputs and absolutely continuous functions on $[0,T]$. 
The optimal cost is defined as
\begin{equation}
\label{eq:opt_cost}
J^* = \hspace*{-0.5em} \inf\limits_{(x,u) \in {\cal P}_T} \hspace*{-0.5em} J(x,u).
\end{equation}


\subsection{Liouville's Equation}

To address this problem, we begin by defining measures whose supports model the evolution of families of trajectories. 
An initial condition and its evolution can be understood via Equation \eqref{eq:def:F}, but the evolution of a family of trajectories understood via a measure must be formalized in a different manner:
Let $\LF: C^1\left( [0,T]\times X \right) \rightarrow C\left( [0,T]\times X \times U \right)$ be a linear operator which acts on a test function $v$ as:
\begin{equation}
\label{eq:LF}
\LF v = \frac{\partial v}{\partial t} + \sum_{k=1}^n \frac{\partial v}{\partial x_k}[F(t,x,u)]_k
\end{equation}
Similarly let $\Lf: C^1\left( [0,T]\times X \right) \rightarrow C\left( [0,T]\times X \right)$ be a linear operator which acts on a test function $v$ as:
\begin{equation}
\label{eq:Lf}
\Lf v = \frac{\partial v}{\partial t} + \sum_{k=1}^n \frac{\partial v}{\partial x_k}[f(t,x)]_k
\end{equation}
Finally, let $\Lg: C^1\left( [0,T]\times X \right) \rightarrow C\left( [0,T]\times X \right)^m$ be a linear operator which acts on a test function $v$ as:
\begin{equation}
\label{eq:Lg}
[\Lg v]_j = \sum_{k=1}^n \frac{\partial v}{\partial x_k}[g(t,x)]_{kj}, \quad \forall j \in \{1, \cdots, m\}
\end{equation}


\begin{rem}
Using the dual relationship between measures and functions, let $\LF': C([0,T] \times X \times U)' \rightarrow C^1([0,T] \times X)'$ be the adjoint operator of $\LF$:
\begin{equation}
\langle \LF' \mu, v \rangle = \langle \mu, \LF v \rangle = \hspace*{-0.3cm} \int\limits_{[0,T] \times X \times U} \hspace*{-0.3cm} \LF v(t,x)\, d\mu(t,x,u)
\end{equation}
for all $\mu \in \M_+([0,T]\times X \times U)$ and $v \in C^1([0,T] \times X)$. 
Similarly, we can define $\Lf' : C([0,T] \times X)' \rightarrow C^1([0,T] \times X)'$ and $\Lg' : \left( C([0,T] \times X)^m \right)' \rightarrow C^1([0,T] \times X)'$ as the adjoint operators of $\Lf$ and $\Lg$, respectively.
\end{rem}

Each of these adjoint operators can be used to describe the evolution of families of trajectories of the system. 
To formalize this relationship, consider an admissible trajectory $x$ defined on $t \in [0,T]$, we define its associated \emph{occupation measure}, denoted as $\mu(\cdot \mid x) \in \M_+([0,T] \times X)$, as:
\begin{equation}
\label{eq:def:mu|x}
\mu(A \times B \mid x) = \int_0^T I_{A \times B}( t, x(t))\, dt
\end{equation}
for all subsets $A \times B$ in the Borel $\sigma$-algebra of $[0,T] \times X$, where $I_{A \times B}(\cdot)$ denotes the indicator function on a set $A \times B$. 
The quantity $\mu(A \times B \mid x)$ is equal to the amount of time the graph of the trajectory, $(t,x(t))$, spends in $A \times B$. 
Similarly, the terminal measure associated with $x$ is defined as:
\begin{equation}
\label{eq:def:muT}
\mu_T(B \mid x) = I_B(x(T)),
\end{equation}
where $\mu_T(\cdot \mid x) \in {\cal M}_+(X_T)$.

If the control action $u$ associated with an absolutely continuous function $x$ is also given, the occupation measure associated with the pair $(x,u)$, denoted as $\mu(\cdot \mid x,u) \in \M_+([0,T] \times X \times U)$, can be defined as:
\begin{equation}
\label{eq:def:mu}
\mu(A \times B \times C \mid x,u) = \int_0^T I_{A \times B \times C}(t,x(t),u(t))\, dt
\end{equation}
for all subsets $A \times B \times C$ in the Borel $\sigma$-algebra of $[0,T]\times X\times U$.

\begin{rem}
Let $\mu(\cdot \mid x,u)$ and $\mu_T(\cdot \mid x)$ be the occupation measure and terminal measure associated with the pair $(x,u)$, respectively. By definitions \eqref{eq:def:muT} and \eqref{eq:def:mu}, the cost function can be expressed as:
\begin{equation}
J(x,u) = \langle \mu(\cdot \mid x,u), h \rangle + \langle \mu_T(\cdot \mid x), H \rangle.
\end{equation}
\end{rem}
\noindent Notice that despite the cost function potentially being a nonlinear function for the admissible pair in the space of functions, the analogous cost function over the space of measures is linear. 
In fact, a similar analogue holds true for the dynamics of the system.
That is, the occupation measure associated with an admissible pair satisfy a linear equation over measures:
\begin{lem}
\label{lem:traj2le}
Given an admissible pair $(x,u)$, its occupation measure and terminal measure $(\mu(\cdot \mid x,u), \mu_T(\cdot \mid x))$ satisfy Liouville's Equation which is defined as:
\begin{equation}
v(0,x_0) + \langle \mu(\cdot \mid x,u), \LF v \rangle = \langle \mu_T(\cdot \mid x), v(T,\cdot) \rangle
\end{equation}
for all test functions $v \in C^1([0,T]\times X)$. 
Since this is true for all test function, we write it as a linear operator equation:
\begin{equation}
\label{eq:le}
\delta_{(0,x_0)} + \LF' \mu(\cdot \mid x,u) = \delta_T \otimes \mu_T(\cdot \mid x)
\end{equation}
where $\delta_{(0,x_0)}$ is a Dirac measure at $t=0$, $x = x_0$, $\delta_T$ is a Dirac measure at a point $t=T$, and $\otimes$ denotes the product of measures.
\end{lem}
\begin{proof}
This lemma follows directly from the definition of occupation measure and terminal measure.
\end{proof}

Now one can ask whether the converse relationship holds. 
That is, do measures that satisfy Liouville's Equation correspond to trajectories that satisfy the dynamical system defined in Equation \eqref{eq:def:F}?
To answer this question, we first consider a family of trajectories modeled by a probability measure $\rho \in \M_+({\cal X}_T)$ defined on the space of admissible trajectories and define an average occupation measure $\zeta \in \M_+([0,T] \times X)$ for the family of trajectories as:
\begin{equation}
\label{eq:def:zeta}
\zeta(A \times B) = \int_{{\cal X}_T} \mu(A \times B \mid x(\cdot))\, d\rho(x(\cdot))
\end{equation}
and an average terminal measure, $\zeta_T \in \M_+(X_T)$, by
\begin{equation}
\zeta_T(B) = \int_{{\cal X}_T} \mu_T(B \mid x(\cdot))\, d\rho(x(\cdot))
\end{equation}
The solutions to Liouville's Equation can be marginalized into a conditional measure over the input given a state and time and a marginal distribution over just state and time. 
Each marginal distribution in fact coincides with the average measures, and the corresponding conditional probability distribution corresponds to trajectories of the dynamical system as defined in Equation \eqref{eq:def:F}:
\begin{lem}
\label{lem:le2traj}
Let $(\mu, \mu_T)$ be a pair of measures satisfying Liouville's Equation \eqref{eq:le} such that $\spt(\mu) \subset [0,T]\times X \times U$, and $\spt(\mu_T) \subset X_T$. 
Then
\begin{enumerate}[(i)]
\item The measure $\mu$ can be disintegrated as
\begin{equation}
    d\mu(t,x,u) = d\nu(u \mid t,x)\, d\bar{\mu}(t,x)
    = d\nu(u \mid t,x)\, d\mu(x \mid t)\, dt
\end{equation}
where $\bar{\mu}$ is the $(t,x)$-marginal of $\mu$, $\nu(\cdot \mid t,x)$ is a stochastic kernel on $U$ given $(t,x) \in [0,T]\times X$, $\mu(\cdot \mid t)$ is a stochastic kernel on $X$ given $t \in [0,T]$, and $\mu(\cdot \mid T) = \mu_T$.
\item There exists a non-negative probability measure $\rho \in \M_+({\cal X}_T)$ supported on a family of absolutely continuous admissible trajectories $\theta(\cdot)$ that satisfy the differential equation:
\begin{equation}
\label{eq:F(t,x,U)}
\dot{\theta}(t) = \int_U F(t,\theta(t),u)\, d\nu(u \mid t,\theta(t)) \in F(t,\theta(t),U)
\end{equation}
almost everywhere, such that for all measurable functions $w:\R^n \rightarrow \R$ and $t \in [0,T]$
\begin{equation}
\label{eq:rho}
\int_X w(x)\, \mu(x \mid t) = \int_{{\cal X}_T} w(\theta(t))\, d\rho(\theta(\cdot)).
\end{equation}
\item The family of trajectories in the support of $\rho$ starts from $x_0$ at $t=0$. 
Moreover, the average occupation measure, $\zeta$, and average terminal measure, $\zeta_T$, generated by this family of trajectories coincide with $\bar{\mu}$ and $\mu_T$, respectively.
\end{enumerate}
\end{lem}

\begin{proof}
Lemma \ref{lem:le2traj} follows from the proof of \cite[Lemma~3]{henrion2014}, applied to the case of $\delta_0 \otimes \mu_0 = \delta_{(0,x_0)}$ (i.e., $\mu_0$ is a Dirac measure at $x=x_0$). Equation \eqref{eq:F(t,x,U)} is due to Equation \eqref{eq:def:F} and \eqref{eq:def:U}. $\rho$ is a probability measure because $\int_{{\cal X}_T}\, d\rho = 1$, which follows from plugging $w(x)=1$ in Equation \eqref{eq:rho}.
\end{proof}

As a result, the support of the measures that satisfy Liouville's Equation \eqref{eq:le} coincide with trajectories that satisfy the differential equation defined in Equation \eqref{eq:F(t,x,U)}. 
Moreover the solutions to Equations \eqref{eq:def:F} and \eqref{eq:F(t,x,U)}, as shown in \cite[Corollary~3.2]{frankowska2000filippov}, are identical:
\begin{rem}
\label{rem:udnu}
For any $\theta(\cdot) \in \spt(\rho)$, Equation \eqref{eq:F(t,x,U)} can be rewritten as:
\begin{equation}
\dot{\theta}(t) = f(t,\theta(t)) + g(t,\theta(t)) \int_U u\, d\nu(u \mid t,\theta(t))
\end{equation}
Since $\nu(\cdot \mid t,x)$ is a stochastic kernel, we have $\int_U u\, d\nu(u \mid t,\theta(t)) \in U$. 
Therefore $(\int_U u\, d\nu(u \mid \cdot,\theta(\cdot)),\theta(\cdot))$ is an admissible pair.
\end{rem}

Lemma \ref{lem:traj2le} and \ref{lem:le2traj} states the correspondence between admissible trajectories and measures that satisfy Liouville's Equation \eqref{eq:le}. 
This correspondence allows us to search for solutions in the space of measures using a convex formulation of the optimal control problem.


  \section{Infinite Dimensional Linear Program }
\label{sec:LP}

This section formulates the $(OCP)$ as an infinite dimensional linear program over the space of measures, proves that it computes the solution to $(OCP)$, and illustrates how its solution can be used for control synthesis. 

Define an infinite dimensional linear program $(P)$ as:
\begin{flalign} 
&& \inf_\Gamma & \phantom{4} \langle \mu, h \rangle + \langle \mu_T, H \rangle && (P) \nonumber \\
&& \text{s.t.} & \phantom{4} \delta_{(0,x_0)} + \LF' \mu  = \delta_T \otimes \mu_T, \nonumber \\
&&& \phantom{4} \LF' \mu = \Lf' \gamma + \Lg'( \sigma_+ - \sigma_- ), \nonumber \\
&&& \phantom{4} [\sigma_+]_j + [\sigma_-]_j + [\hat{\sigma}]_j = \gamma &&\forall j \in \{1,\cdots,m\}, \nonumber \\
&&& \phantom{4} \gamma = \pi_*^{(t,x)} \mu, \nonumber\\
&&& \phantom{4} \mu, \gamma, \mu_T \geq 0, \nonumber\\
&&& \phantom{4} [\sigma_+]_j, [\sigma_-]_j, [\hat{\sigma}]_j \geq 0 && \forall j \in \{1,\cdots,m\}, \nonumber
\end{flalign}
where the infimum is taken over a tuple of measures $\Gamma = (\mu, \mu_T, \gamma, \sigma_+, \sigma_-, \hat{\sigma}) \in \M_+([0,T]\times X \times U) \times \M_+(X_T) \times \M_+([0,T] \times X) \times \left(\M_+([0,T] \times X)\right)^m \times \left(\M_+([0,T] \times X)\right)^m \times \left(\M_+([0,T] \times X)\right)^m$.
The dual to problem $(P)$ is given as:
\begin{flalign} 
&& \sup_\Delta & \phantom{4} v(0,x_0) && (D) \nonumber \\
&& \text{s.t.} & \phantom{4} \LF v + \LF w + q + h \geq 0, \nonumber \\
&&& \phantom{4} v(T,x) \leq H(x), && \forall x \in X_T \nonumber \\
&&& \phantom{4} \Lf w + \sum_{j=1}^m p_j + q \leq 0, \nonumber\\
&&& \phantom{4} \left| \left[\Lg w\right]_j \right| \leq p_j && \forall j \in \{1,\cdots, m\}, \nonumber\\
&&& \phantom{4} p_j \geq 0 && \forall  j \in \{1,\cdots,m\}. \nonumber
\end{flalign}
where the supremum is taken over a tuple of functions $\Delta = (v,w,p,q) \in C^1([0,T]\times X) \times C^1([0,T]\times X) \times \left(C([0,T]\times X)\right)^{m} \times C([0,T]\times X)$.

Next, we have the following useful result:
\begin{thm}
There is no duality gap between $(P)$ and $(D)$.
\end{thm}

\begin{proof}
The proof follows from \cite[Theorem~3.10]{anderson1987}.
\end{proof}

Next, we show that $(P)$ solves $(OCP)$. 
We do this by first introducing another optimization problem $(Q)$ that solves $(OCP)$, and then showing that $(Q)$ is equivalent as $(P)$.
Define optimization problem $(Q)$ as
\begin{flalign} 
&& \inf_\Theta & \phantom{4} \langle \eta, h \rangle + \langle \eta_T, H \rangle && (Q) \nonumber \\
&& \text{s.t.} & \phantom{4} \delta_{(0,x_0)} + \LF' \eta = \delta_T \otimes \eta_T \nonumber \\
&&& \phantom{4} \eta^i, \eta_T^i \geq 0 \nonumber
\end{flalign}
where the inifimum is taken over a tuple of measures $\Theta = (\eta, \eta_T) \in \M_+([0,T]\times X \times U) \times \M_+(X_T)$.

Note that $(Q)$ is identical to the primal LP defined in Equation $(2.13)$ in \cite{lasserre2008}, which was shown to compute the optimal value of $(OCP)$ under the following assumption:
\begin{assum}
$h(t,x,\cdot)$ is convex for any $(t,x) \in [0,T]\times X$.
\end{assum}
Unfortunately, control synthesis was not amenable using $(Q)$. 
In contrast our formulation, as we describe next, makes control synthesis feasible under the following additional assumption:
\begin{assum}
\label{assum:unique}
If $(OCP)$ is feasible, then the optimal admissible pair $(x^*,u^*)$ is unique $dt$-almost everywhere.
\end{assum}
\begin{rem}
Notice that Assumption \ref{assum:unique} and Equation \eqref{eq:opt_cost} imply $J^* = J(x^*,u^*)$.
\end{rem}

Next, we prove several important properties about $(Q)$:
\begin{lem}
\label{lem:Q}
If $(Q)$ is feasible, then
\begin{enumerate}[(i)]
	\item the minimum to $(Q)$, $q^*$, is attained,
	\item $q^* = J^* = J(x^*,u^*)$, where $(x^*,u^*)$ is the optimal admissible pair to $(OCP)$, 
    \item if $(\eta^*, \eta^*_T)$ is an optimal solution to $(Q)$, we can disintegrate $\eta^*$ as
    \begin{equation}
    \label{eq:decomp}
    \begin{aligned}
    d\eta^*(t,x,u) &= d\nu^*(u \mid t,x)\, d\bar{\eta}^*(t,x) \\
                   &= d\nu^*(u \mid t,x)\, d\eta^*(x \mid t)\, dt,
    \end{aligned}
    \end{equation}
    then $\bar{\eta}^*$ coincides with the occupation measures associated with $x^*$ almost everywhere,
    \item for almost every point $(t,x)$ in the support of $\bar{\eta}^*$, we have
    \begin{equation}
    \label{eq:lemmaQ41}
    F(t,x,u^*(t)) = \int_U F(t,x,u)\, d\nu^*(u \mid t,x),
    \end{equation}
    and if every column of $g$ is nonzero almost everywhere along the optimal trajectory trajectory $x^*$, then
    \begin{equation}
    \label{eq:lemmaQ42}
    u^*(t) = \int_U u \, d\nu^*(u \mid t,x).
    \end{equation}
    $\bar{\eta}^*$-almost everywhere.
\end{enumerate}
\end{lem}

\begin{proof}
\hfill
\begin{enumerate}[(i)]
\item This follows from \cite[Theorem~2.3(i)]{lasserre2008}.
\item This follows from \cite[Theorem~2.3(iii)]{lasserre2008} by noting $F(t,x,u) = f(t,x) + g(t,x)u$ and $h(t,x,\cdot)$ is convex.
\item We first show that $\bar{\eta}^*$ coincides with the occupation measure of a family of trajectories, and then argue that any of the trajectories together with some control will achieve optimal cost. The result then follows by noting the optimal admissible pair is unique almost everywhere.

Since $(\eta^*,\eta_T^*)$ is optimal therefore feasible, it satisfies Liouville's Equation \eqref{eq:le}. Thus $\eta^*$ can be disintegrated according to Lemma \ref{lem:le2traj}$(i)$. By Lemma \ref{lem:le2traj}$(ii)$ and $(iii)$, there exists a probability measure $\rho \in \M_+({\cal X}_T)$ such that $\bar{\eta}^*$ coincides with the occupation measures of a family of admissible trajectories in the support of $\rho$. We only need to show all the trajectories in that family are equal to $x^*(t)$ $dt$-almost everywhere.

Define $\hat{u}(t,x) = \int_U u\, d\nu(u \mid t,x)$, we have
\begin{align}
J^* = &\int_{[0,T]\times X\times U} h(t,x,u)\, d\eta^*(t,x,u) + \int_{X_T} H(x)\, d\eta^*_T(x) \nonumber\\
=& \int_{[0,T]\times X\times U} h(t,x,u)\, d\nu^*(u \mid t,x)\, d\eta^*(x \mid t)\, dt + \int_{X_T} H(x)\, d\eta^*_T(x) \nonumber\\
\geq& \int_{[0,T]} \int_X h(t,x,\hat{u}(t,x))\, d\eta^*(x \mid t)\, dt + \int_{X_T} H(x)\, d\eta^*_T(x) \label{arg:hconvex}\\
=& \int_{[0,T]} \int_{{\cal X}_T} h(t,\theta(t), \hat{u}(t,\theta(t)))\, d\rho(\theta(\cdot))\, dt + \int_{{\cal X}_T} H(\theta(T))\, d\rho(\theta(\cdot)) \label{arg:rho}\\
=& \int_{{\cal X}_T} \left[ \int_{[0,T]} h(t,\theta(t), \hat{u}(t,\theta(t)))\, dt + H(\theta(T)) \right]\, d\rho(\theta(\cdot)) \label{arg:fubini}\\
=& \int_{{\cal X}_T} J \left(\theta(\cdot), \hat{u}(\cdot, \theta(\cdot)) \right) d\rho(\theta(\cdot)) \label{arg:admissible_pair}
\end{align}
where \eqref{arg:hconvex} is obtained from the convexity of $h(t,x,\cdot)$ and the fact that $\nu^*(\cdot \mid t,x)$ is a probability measure; \eqref{arg:rho} is from Lemma \ref{lem:le2traj}; \eqref{arg:fubini} is from Fubini's Theorem; \eqref{arg:admissible_pair} is because $(\theta(\cdot), \hat{u}(\cdot, \theta(\cdot)))$ is an admissible pair.

Since $\rho$ is probability measure, every admissible pair $(\theta(\cdot),\hat{u}(\cdot,\theta(\cdot)))$ with $\theta(\cdot)$ in the support of $\rho$ must be optimal. Since $x^*$ is assumed to be unique $dt$-almost everywhere, we have
\begin{equation}
\label{eq:gamma2xstar}
\theta(t) = x^*(t)\text{ a.e.}\quad \forall \theta(\cdot) \in \spt(\rho)
\end{equation}

According to Lemma \ref{lem:le2traj}, $\bar{\eta}^*$ coincides with the occupation measure of the family of admissible trajectories in $\spt(\rho)$. Note the similarity between Equation \eqref{eq:def:mu|x} and \eqref{eq:def:zeta}, therefore 
$\bar{\eta}^*$ coincides with the occupation measure of $x^*$ almost everywhere.
\item
From Lemma \ref{lem:le2traj} and the proof of $(iii)$, we know for any trajectory $\theta(\cdot)$ in the support of $\rho$,
\begin{equation}
\dot{\theta}(t) = \int_U F(t,\theta(t),u)\, d\nu^*(u \mid t,\theta(t)) \text{ a.e.}
\end{equation}
Using Equation \eqref{eq:gamma2xstar} and the fact that
\begin{equation}
\dot{x}^*(t) = F(t,x^*(t),u^*(t)) \text{ a.e.},
\end{equation}
Equation \eqref{eq:lemmaQ41} follows.

Since $F(t,x,u) = f(t,x) + g(t,x)u$, we know
\begin{equation}
g(t,x) u^*(t) = g(t,x) \int_U u\, d\nu^*(u \mid t,x) \quad \bar{\eta}^* \text{-a.e.}
\end{equation}
Since $g(t,x^*(t))$ is nonzero almost everywhere and therefore $g(t,x)$ is nonzero $\bar{\eta}^*$-almost everywhere, Equation \eqref{eq:lemmaQ42} follows.
\end{enumerate}
\end{proof}

The previous result ensures that $(Q)$ can be solved to find a solution to $(OCP)$ in a convex manner; however, control synthesis via this formulation is still not amenable. 
The next pair of results ensure that $(P)$ can be used to solve $(Q)$ and perform control extraction:
\begin{thm}
\label{thm:P}
\hfill
\begin{enumerate}[(i)]
\item $(P)$ is feasible if and only if $(Q)$ is feasible. 
Furthermore, if $(P)$ is feasible the minimum to $(P)$, $p^*$, is attained, and $p^* = J^*$.
\item If $(P)$ is feasible, then let $(\mu^*,\mu_T^*, \gamma^*, \sigma_+^*, \sigma_-^*, \hat{\sigma}^*)$ be a minimizer of $(P)$, then $\gamma^*$ coincides with the occupation measure of $x^*$ almost everywhere.
\end{enumerate}
\end{thm}

\begin{proof}
\hfill
\begin{enumerate}[(i)]
\item Given any feasible point $(\mu,\mu_T, \gamma, \sigma_+, \sigma_-, \hat{\sigma})$ of $(P)$, clearly $(\mu,\mu_T)$ will also be feasible in $(Q)$ with the same cost. Therefore $(Q)$ is feasible and $p^* \geq q^*$, where $q^*$ is the minimum to $(Q)$. Since $q^* = J^*$, we only need to show there is a feasible point in $(P)$ that achieves $q^*$ as the cost.

To simplify the notations, we assume the number of inputs $m = 1$. The case of $m>1$ can be proved by a similar argument. Let $(\eta^*, \eta^*_T)$ be the optimal solution of $(Q)$, then the Liouville's Equation \eqref{eq:le} is automatically satisfied. Again, we disintegrate $\eta^*$ as in Equation \eqref{eq:decomp}, and define $\gamma := \pi_*^{(t,x)}\eta^* = \bar{\eta}^*$. Since $u^*$ is assumed to be measurable, according to Riesz representation theorem, there exists signed measures $\sigma \in \M([0,T]\times X)$ such that
\begin{equation}
\label{eq:sigma}
d \sigma(t,x) = u^*(t) \, d \gamma(t,x)
\end{equation}
Furthermore, since $u^*(t) \in [-1,1]$, we have $-\gamma \leq \sigma \leq \gamma$ on support of $\gamma$. Then we use Hahn-Jordan decomposition to express $\sigma(t,x)$ as
\begin{equation}
\label{eq:hahn}
\sigma = \sigma_+ - \sigma_-
\end{equation}
where $\sigma_+, \sigma_- \in \M_+([0,T]\times X)$ and $\sigma_+ + \sigma_- \leq \gamma$. Thus there exists a measure $\hat{\sigma} \in \M_+([0,T]\times X)$ such that
\begin{equation}
\label{eq:sigma_hat}
\sigma_+ + \sigma_- + \hat{\sigma} = \gamma
\end{equation}
For any function $v \in C^1([0,T]\times X)$, we have
\begin{equation}
\label{eq:L2Lf}
\begin{aligned}
\int_{[0,T]\times X \times U} \LF v \, d\eta^* =& \int_{[0,T]\times X \times U} \Lf v \, d\eta^* + \int_{[0,T]\times X \times U} (\Lg v) \cdot u \, d\nu^*(u \mid t,x) \, d\gamma \\
=& \int_{[0,T]\times X} \Lf v \, \int_U d\eta^* + \int_{[0,T]\times X} \Lg v \left[ \int_U u \, d\nu^*(u\mid t,x)\right] d\gamma\\
=& \int_{[0,T]\times X} \Lf v \, d(\pi_*^{(t,x)} \eta^*)+ \int_{[0,T]\times X} (\Lg v) u^*(t) \, d\gamma\\
=& \int_{[0,T]\times X} \Lf v \, d\gamma + \int_{[0,T]\times X} \Lg v \, d(\sigma_+ - \sigma_-)
\end{aligned}
\end{equation}
Now that we've shown $(\eta^*, \eta_T^*, \gamma, \sigma_+, \sigma_-, \hat{\sigma})$ satisfy all the constraints in $(P)$ and achieves the cost $q^*$, it follows that $p^* = q^*$.

\item This follows from Lemma \ref{lem:Q}$(iii)$ and the proof of $(i)$.

\end{enumerate}
\end{proof}

Finally we describe how to perform control synthesis with the solution to $(P)$:
\begin{thm}
\label{thm:lp}
Suppose $(P)$ is feasible and let $\Gamma^* = (\mu^*, \mu_T^*, \gamma^*, \sigma_+^*, \sigma_-^*, \hat{\sigma}^*)$ be the vector of measures that achieves the infimum of $(P)$, then there exists a control law, $\tilde{u} \in L^1([0,T]\times \mathcal{D}, U)$, such that 
\begin{equation}
\label{eq:radon}
\int_{A\times B} [\tilde{u}(t,x)]_j \, d\gamma^*(t,x) = \int_{A \times B} d[\sigma_+^* - \sigma_-^*]_j(t,x)
\end{equation}
for all subsets $A \times B$ in the Borel $\sigma$-algebra of $[0,T]\times X$ and for each $j \in \{1, \cdots , m\}$. If moreover $(x^*,u^*)$ are optimal solutions to $(OCP)$ and columns of $g(t,x^*(t))$ are nonzero almost everywhere along the optimal trajectory $x^*(t)$ (e.g. it is sufficient for any element in each column of $g$ to be nonzero almost everywhere), then $\tilde{u}$ and $u^*$ are equal $\gamma^*$-almost everywhere.
\end{thm}

\begin{proof}
We will prove the first result using Radon-Nikodym theorem, and the second result can be shown by arguing $\langle \gamma^*, \psi \cdot [\tilde{u}-u^*] \rangle = 0$ for any $\psi$ in some dense subset of $L^1([0,T]\times X)$.

$[\sigma_+^*]_j$, $[\sigma_-^*]_j$, and $\gamma^*$ are $\sigma$-finite for all $j \in \{1, \cdots, m\}$ since they are Radon measures defined over a compact set. Define $[\sigma^*]_j = [\sigma_+^*]_j - [\sigma_-^*]_j$ and notice that each $[\sigma^*]_j$ is also $\sigma$-finite. Since $[\sigma_+^*]_j + [\sigma_-^*]_j + [\hat{\sigma}^*]_j = \gamma^*$ and $[\sigma_+^*]_j, [\sigma_-^*]_j, [\hat{\sigma}^*]_j \geq 0$, $\sigma^*$ is absolutely continuous with respect to $\gamma^*$. Therefore as a result of Radon-Nikodym theorem, there exists a $\tilde{u} \in L^1([0,T]\times X, U)$, which is unique $\gamma^*$-almost everywhere, that satisfies Equation \eqref{eq:radon}.

Since $(\mu^*, \mu_T^*, \gamma^*, \sigma_+^*, \sigma_-^*, \hat{\sigma}^*)$ is feasible in $(P)$, we have
\begin{equation}
\LF'\mu^* = \Lf' \gamma^* + \Lg' (\sigma_+^* - \sigma_-^*)
\end{equation}
By the same argument in Equation \eqref{eq:L2Lf}, we can get
\begin{equation}
\begin{aligned}
\int_{[0,T]\times X} \frac{\partial v}{\partial x} g(t,x) \cdot u^*(t) \, d\gamma^* = \int_{[0,T]\times X} \frac{\partial v}{\partial x} g(t,x) \cdot \tilde{u}(t,x) \, d\gamma^*, \quad \forall v \in C^1([0,T]\times X)
\end{aligned}
\end{equation}
or equivalently
\begin{equation}
\label{eq:u_converge}
\int_{[0,T]\times X} \frac{\partial v}{\partial x} g(t,x) \cdot [\tilde{u}(t,x) - u^*(t)] \, d\gamma^* = 0 \quad \forall v \in C^1([0,T]\times X)
\end{equation}
Since $\left\{\frac{\partial v}{\partial x}, v \in C^1([0,T]\times X)\right\} \supset C^1([0,T]\times X)$, Equation \eqref{eq:u_converge} implies
\begin{equation}
\int_{[0,T]\times X} \psi(t,x)\, g(t,x) \cdot [\tilde{u}(t,x) - u^*(t)] \, d\gamma^* = 0 \quad \forall \psi \in C^1([0,T]\times X)
\end{equation}
Since $g(t,x)$ is nonzero $\gamma^*$-almost everywhere, and $C^1([0,T]\times X)$ is dense in $L^1([0,T]\times X)$, this implies $\tilde{u} = u^*$ $\gamma^*$-almost everywhere.
\end{proof}

  \section{Numerical Implementation}
\label{sec:implementation}

We compute a solution to the infinite-dimensional problem $(P)$ via a sequence of finite dimensional approximations formulated as Semi-Definite Programs (SDP)s. 
These are generated by discretizing the measures in $(P)$ using moments and restricting the functions in $(D)$ to polynomials.
The solutions to any of the SDPs in this sequence can
be used to synthesize an approximation to the optimal controllers.
A comprehensive introduction to such moment relaxations
can be found in \cite{lasserre2009}.

To derive this discretization, we begin with a few preliminaries.
Let $\R_k[x]$ denote the space of real valued multivariate polynomials of total degree less than or equal to $k$. 
Then any polynomial $p(x) \in \R_k[x]$ can be expressed in the monomial basis as:
\begin{equation}
p(x) = \sum_{|\alpha| \leq k} p_\alpha x^\alpha = \sum_{|\alpha| \leq k} p_\alpha \cdot (x_1^{\alpha_1}\cdots x_n^{\alpha_n} )
\end{equation}
where $\alpha$ ranges over vectors of non-negative integers such that $|\alpha| = \sum_{i=1}^n \alpha_i \leq k$, and we denote $\text{vec}(p) = (p_\alpha)_{|\alpha| \leq k }$ as the vector of coefficients of $p$.
Given a vector of real numbers $y = (y_\alpha)$ indexed by $\alpha$, we define the linear functional $L_y: \R_k[x] \rightarrow \R$ as:
\begin{equation}
L_y(p) := \sum_\alpha p_\alpha y_\alpha
\end{equation}
Note that, when the entries of $y$ are moments of a measure $\mu$:
\begin{equation}
y_\alpha = \int x^\alpha \, d\mu(x),
\end{equation}
then
\begin{equation}
    \langle \mu, p \rangle = \int \left( \sum_\alpha p_\alpha x^\alpha \right)\, d\mu = L_y(p).
\end{equation}
If $|\alpha| \leq 2k$, the moment matrix $M_k(y)$ is defined as
\begin{equation}
[M_k(y)]_{(\alpha,\beta)} = y_{(\alpha + \beta)}
\end{equation}
Given any polynomial $h \in \R_{l}[x]$ with $l < k$, the localizing matrix $M_k(h,y)$ is defined as
\begin{equation}
[M_k(h,y)]_{(\alpha,\beta)} = \sum_{|\gamma| \leq l} h_\gamma y_{(\gamma + \alpha + \beta)}.
\end{equation}
Note that the moment and localizing matrices are symmetric and linear in moments $y$.

\subsection{LMI Relaxations and SOS Approximations}

To construct approximations to $(P)$, we make the following additional assumptions:
\begin{assum}
The functions $h$, $H$, components of $f$ and $g$ are polynomials.
\end{assum}

We further assume that the state constraints and target set are described using polynomials:
\begin{assum}
$X$ and $X_T$ are defined as semi-algebraic sets:
\begin{equation}
\begin{aligned}
X = \left\{ x \in \R^n \mid h_{X_i} \geq 0, \, \forall i \in \{1, \cdots, n_X\} \right\}\\
X_T = \left\{ x \in \R^n \mid h_{T_i} \geq 0, \, \forall i \in \{1, \cdots, n_T\} \right\}
\end{aligned}
\end{equation}
where $h_{X_i}, h_{T_i} \in \R[x]$.
\end{assum}


 An sequence of SDPs approximating $(P)$ can be obtained by replacing constraints on measures with constraints on moments.
 Since $h(t,x,u)$ and $H(x)$ are polynomials, the objective function of $(P)$ can be written using linear functionals as $L_{y_\mu}(h) + L_{y_{\mu_T}}(H)$. 
 The equality constraints in $(P)$ can be approximated by an infinite-dimensional linear system, which is obtained by restricting to polynomial test functions: $v(t,x) = t^\alpha x^\beta$, $w(t,x) = t^\alpha x^\beta$, $[p]_j(t,x) = t^\alpha x^\beta$, and $q(t,x) = t^\alpha x^\beta$, for any $\alpha \in \N$, $\beta \in \N^n$. 
 For example, Liouville's Equation \eqref{eq:le} is written with respect to moments as:
\begin{equation}
0^\alpha x_0^\beta + L_{y_\mu} (\LF (t^\alpha x^\beta)) = L_{y_{\mu_T}} (T^\alpha x^\beta)
\end{equation}
The positivity constraints in $(P)$ can be replaced with semi-definite constraints on moment matrices and localizing matrices as discussed above.

A finite-dimensional SDP is then obtained by truncating the degree of moments and polynomial test functions to $2k$. 
Let $\Gamma = (\mu,\mu_T,\gamma,\sigma_+,\sigma_-,\hat{\sigma})$, and let $(y_{k,\xi})$ be the sequence of moments truncated to degree $2k$ for each $\xi \in \Gamma$, and let $\mathbf{y}_k$ be a vector of all the sequences $(y_{k,\xi})$. 
The finite-dimensional linear system is then represented by the linear system:
\begin{equation}
A_k(\mathbf{y}_k) = b_k
\end{equation}
Define the $k$-th relaxed SDP representation of $(P)$, denoted $(P_k)$, as:
\begin{flalign} 
&& \inf_{\mathbf{y}_k} & \phantom{4} L_{y_{k,\mu}}(h) + L_{y_{k,\mu_T}}(H) && (P_k) \nonumber \\
&& \text{s.t.} & \phantom{4} A_k(\mathbf{y}_k) = b_k, \nonumber\\
&&& \phantom{4} M_k(y_{k,\xi}) \succeq 0 && \forall \xi \in \Gamma, \nonumber\\
&&& \phantom{4} M_{k_{X_i}}(h_{X_i},y_{k,\xi}) \succeq 0 && \forall (i,\xi) \in \{1,\cdots,n_X\} \times \Gamma\backslash \{\mu_T\}, \nonumber\\
&&& \phantom{4} M_{k-1}(h_{U_j},y_{k,\mu}) \succeq 0 && \forall j \in \{1,\cdots,m\}, \nonumber\\
&&& \phantom{4} M_{k_{T_i}}(h_{T_i},y_{k,\mu_T}) \succeq 0 && \forall i \in \{1,\cdots,n_T\}, \nonumber\\
&&& \phantom{4} M_{k-1}(h_\tau, y_{k,\xi}) \succeq 0 && \forall \xi \in \Gamma \backslash \{\mu_T\} \nonumber
\end{flalign}
where the infimum is taken over the sequences of moments: $(y_{k,\xi})$ for each $\xi \in \Gamma$, $h_\tau(t) = t(T-t)$, $h_{U_j}(u) = 1 - [u]_j^2$, $k_{X_i} = k - \lceil \text{deg}(h_{X_i})/2 \rceil$, $k_{T_i} = k - \lceil \text{deg}(h_{T_i})/2 \rceil$, and $\succeq$ denotes positive semi-definiteness of matrices.

The dual of $(P_k)$ is a Sums-of-Squares (SOS) program denoted by $(D_k)$ for each $k \in \N$, which is obtained by first restricting the optimization space in $(D)$ to the polynomial functions with degree truncated to $2k$ and by then replacing the non-negativity constraints in $(D)$ with SOS constraints. 
Define $Q_{2k}(h_{T_1}, \cdots, h_{T_{n_T}}) \subset \R_{2k}[x]$ to be the set of polynomials $l \in \R_{2k}[x]$ expressible as
\begin{equation}
l = s_0 + \sum_{i=1}^{n_T} s_i h_{T_i}
\end{equation}
for some polynomials $\{s_i\}_{i=0}^{n_T} \subset \R_{2k}[x]$ that are sums of squares of other polynomials. 
Every such polynomial is clearly non-negative on $X_T$. 
Similarly, we define $Q_{2k}(h_\tau, h_{X_1}, \cdots, h_{X_{n_X}}, h_{U_1}, \cdots, h_{U_m}) \subset \R_{2k}[t,x,u]$, and $Q_{2k}(h_\tau, h_{X_1}, \cdots, h_{X_{n_X}}) \subset \R_{2k}[t,x]$.
The $k$-th relaxed SDP representation of $(D)$, denoted $(D_k)$, is then given as
\begin{flalign} 
 \sup_{(v,w,p,q)} & \phantom{4} v(0,x_0) \hspace{4cm} (D_k) \nonumber \\
 \text{s.t.}\hspace{0.3cm} & \phantom{4} \LF(v+w) + q + h \in Q_{2k}(h_\tau, h_{X_1}, \cdots, h_{X_{n_X}}, h_{U_1}, \cdots, h_{U_m}), \nonumber\\
& \phantom{4} -v(T,\cdot) + H \in Q_{2k}(h_{T_1}, \cdots, h_{T_{n_T}}), \nonumber\\
& \phantom{4} -\Lf w - \sum_{j=1}^m p_j - q \in Q_{2k}(h_\tau, h_{X_1}, \cdots, h_{X_{n_X}}), \nonumber\\
& \phantom{4} p - \left(\Lg w \right)^\text{T} \in \left(Q_{2k}(h_\tau, h_{X_1}, \cdots, h_{X_{n_X}})\right)^m, \nonumber\\
& \phantom{4} p + \left(\Lg w \right)^\text{T} \in \left(Q_{2k}(h_\tau, h_{X_1}, \cdots, h_{X_{n_X}})\right)^m, \nonumber\\
& \phantom{4} p \in \left(Q_{2k}(h_\tau, h_{X_1}, \cdots, h_{X_{n_X}})\right)^m \nonumber
\end{flalign}
where the supremum is taken over a tuple of polynomials $(v,w,p,q) \in \R_{2k}[t,x] \times \R_{2k}[t,x] \times \left(\R_{2k}[t,x]\right)^m \times \R_{2k}[t,x]$.

We can first prove that these pair of problems are well-posed:
\begin{thm}
\label{thm:duality_Pk}
For each $k \in \mathbf{N}$, there is no duality gap between $(P_k)$ and $(D_k)$.
\end{thm}
\begin{proof}
This can be proved using Slater's condition (see \cite{boyd2004convex}), which involves noting that $(D_k)$ is bounded below, and then arguing the feasible set has an interior point.
\end{proof}

In fact, the optimal control can be synthesized from the solution as follows:
\begin{rem}
\label{rem:extract}
We can extract a polynomial control law from the solution $y_k^*$ of $(P_k)$ in the following way: given moment sequences truncated to $2k$, we want to find an appropriate control law $u_k^*$ with components $[u_k^*]_j \in \mathbb{R}[t,x]$ such that the analogue of Equation \eqref{eq:radon}:
\begin{equation}
\int\limits_{[0,T]\times X} \hspace*{-.25cm} t^{\alpha_0}x^{\alpha_1} \, [u_k^*]_j(t,x) \, d\gamma(t,x) =\int\limits_{[0,T]\times X} \hspace*{-.25cm} t^{\alpha_0}x^{\alpha_1} \, d( [\sigma_+]_j - [\sigma_-]_j )
\end{equation}
 is satisfied for any $j \in \{1,\cdots, m\}$, and $\sum_{l=0}^n \alpha_l \leq k$, $\alpha_l \geq 0$. 
 When constructing a polynomial control law from the solution of $(P_k)$, these linear equations written with respect to the coefficients of $[u_k^*]_j$ are expressible in terms of $y_{k,\sigma^+}^*$, $y_{k,\sigma^-}^*$, and $y_{k,\gamma}^*$. Direct calculation shows this linear system of equations is:
\begin{equation}
\label{eq:ctrl_extract}
M_k(y_{k,\gamma}^*) \text{vec}([u_k^*]_j) = y_{k, [\sigma_+]_j}^* - y_{k, [\sigma_-]_j}^*
\end{equation}
To extract the coefficients of the control law, one needs only to compute the generalized inverse of $M_k(y_{k,\gamma}^*)$, which exists since $M_k(y_{k,\gamma}^*)$ is positive semidefinite. 
\end{rem}

Note that the degree of the extracted polynomial control law is dependent on the relaxation order $k$. Higher relaxation orders lead to higher degree controllers.

\subsection{Convergence of Relaxed Problems}

Next we can prove the convergence of the pair of approximations:
\begin{thm}
Let $p_k^*$ and $d_k^*$ denote the infimum of $(P_k)$ and supremum of $(D_k)$, respectively. Then $\{p_k^*\}_{k=1}^\infty$ and $\{d_k^*\}_{k=1}^\infty$ converge monotonically from below to the optimal value of $(P)$ and $(D)$.
\end{thm}

\begin{proof}
This theorem can be proved using a similar technique adopted in the proof of \cite[Theorem~4.2]{majumdar2014}. We first establish a lower found of $d_k^*$ by finding a feasible solution to $(D_k)$ for some $k$, and then show that there exists a convergent subsequence of $\{d_k^*\}_{k=1}^\infty$, by arguing the lower bound can be arbitrarily close to $d^*$ for large enough $k$.

Using Theorem \ref{thm:duality_Pk}, we only need to prove $\{d_k^*\}_{k=1}^\infty$ converge monotonically from below to $d^*$.

Note that the higher the relaxation order $k$, the looser the constraint set of the optimization problem $(D_k)$, so $\{d_k^*\}_{k=1}^\infty$ is non-decreasing.

Suppose $(v,w,p,q) \in C^1([0,T]\times X) \times C^1([0,T]\times X) \times C([0,T]\times X)^{m} \times C([0,T]\times X)$ is feasible in $(D)$. For every $\epsilon > 0$, set
\begin{equation}
\begin{aligned}
\tilde{v}(t,x) &:= v(t,x) + 3\epsilon t - (1+3T)\epsilon,\\
\tilde{w}(t,x) &:= w(t,x) - \epsilon t,\\
\tilde{p}_j(t,x) &:= p_j(t,x) + \epsilon / m, \quad \forall j \in \{1,\cdots,m\},\\
\tilde{q}(t,x) &:= q(t,x) - \epsilon
\end{aligned}
\end{equation}
Therefore, $\LF \tilde{v} = \LF v + 3\epsilon$, $\tilde{v}(T,x) = v(T,x) - \epsilon$, $\LF \tilde{w} = \LF w - \epsilon$, $\Lf \tilde{w} = \Lf w - \epsilon$, and $\Lg \tilde{w} = \Lg w$, and it follows that $(\tilde{v},\tilde{w},\tilde{p},\tilde{q})$ is strictly feasible in $(D)$ with a margin at least $\epsilon$. Since $[0,T]\times X$ and $X$ are compact, and by a generalization of the Stone-Weierstrass theorem that allows for the simultaneous uniform approximation of a function and its derivatives by a polynomial \cite{hirsch2012}, we are guaranteed the existence of polynomials $\hat{v}$, $\hat{w}$, $\hat{p}_j$, and $\hat{q}$ such that $\|\hat{v} - \tilde{v}\|_\infty < \epsilon$, $\|\LF \hat{v} - \LF \tilde{v}\|_\infty < \epsilon$, $\|\hat{w} - \tilde{w}\|_\infty < \epsilon$, $\|\LF \hat{w} - \LF \tilde{w}\|_\infty < \epsilon$, $\|\Lf \hat{w} - \Lf \tilde{w}\|_\infty < \epsilon$, $\|[\Lg \hat{w}]_j - [\Lg \tilde{w}]_j\|_\infty < \epsilon$, $\|\hat{p}_j - \tilde{p}_j\|_\infty < \epsilon$, and $\|\hat{q} - \tilde{q}\|_\infty < \epsilon$. By Remark \ref{rem:compact} and Putinar's Positivstellensatz \cite[Theorem~2.14]{lasserre2009}, those polynomials are strictly feasible for $(D_k)$ for a sufficiently large relaxation order $k$, therefore $d_k \geq \tilde{v}(0,x_0)$. Also, since $\tilde{v}(0,x_0) = v(0,x_0) - (1+3T)\epsilon$, we have $d_k^* \geq \hat{v}(0,x_0) > v(0,x_0)-(2+3T)\epsilon = d-(2-3T)\epsilon$, where $2+3T < \infty$ is a constant. Using the fact that $d_k^*$ is nondecreasing and bounded above by $d^*$, we know $\{d_k^*\}_{k=1}^\infty$ converges to $d$ from below.
\end{proof}

Finally we can prove that the sequence of controls extracted as the solution to the linear equation \eqref{eq:ctrl_extract} from the sequence of SDPs converges to the optimal control:
\begin{thm}
Let $\{y_{k,\xi}^*\}_{\xi \in \Gamma}$ be an optimizer of $(P_k)$, and let $\{\mu_k^*\}_{k=1}^\infty$ be any sequence of measures such that the truncated moments of $\mu_k^*$ match $y_{k,\mu}^*$. 
In addition, for each $k \in \mathbb{N}$, let $u_k^*$ denote the controller constructed by Equation \eqref{eq:ctrl_extract}, and $\tilde{u}$ is defined as in Theorem \ref{thm:lp}.
Then there exists a subsequence $\{k_l\}_{l=1}^\infty \subset \mathbb{N}$ such that:
\begin{equation}
\int\limits_{[0,T]\times X} \hspace*{-.25cm} v(t,x)\,[u_{k_l}^*]_j(t,x)\,d\mu_{k_l}^*(t,x) \xrightarrow{l\rightarrow\infty}
\int\limits_{[0,T]\times X} \hspace*{-.25cm} v(t,x)\,[\tilde{u}]_j(t,x)\,d\mu^*(t,x)
\end{equation}
for all $v \in C^1([0,T]\times X)$, and each $j \in \{1, \cdots, m\}$.
\end{thm}

\begin{proof}
This is an extension of the result in \cite[Theorem~4.5]{majumdar2014}.
\end{proof}

\subsection{Free Terminal Time}

Though this original problem formulation is powerful, it is sometimes useful to consider the optimal control problem where  the system state has to be driven to $X_T$ before a fixed time $T_0$, and not necessarily remain in $X_T$ afterwards (as opposed to reaching $X_T$ exactly at time $T$). 
We refer to this problem as free terminal time problem.
We adapt the notation $T$ to denote the first time a trajectory reaches $X_T$ (the terminal time), and an admissible pair $(x,u)$ can be redefined by the following conditions:
Given a point $x_0 \in \R^n$ and a $T_0 > 0$, there exists a $T$ satisfying $0<T\leq T_0$, a control $u :[0,T] \to U$, and an absolutely continuous function $x : [0,T] \to \R^n$ such that:
\begin{enumerate}
\item $x(0) = x_0$ and  $x(T) \in X_T$,
\item $x(t) \in X$ for all  $t \in [0,T]$, and
\item $\dot{x}(t) = F(t,x(t),u(t))$ for almost every $t \in [0,T]$.
\end{enumerate}
In practice this formulation requires, we modify $(OCP)$ by adding in another constraint $0<T \leq T_0$ since $T$ is a variable now.

The primal LP that solves the free terminal time is obtained by modifying the support of $\mu_T$ in $(P)$ to be $[0,T_0]\times X_T$, and substituting $\delta_T \otimes \mu_T$ with $\mu_T$ in its first constraint. 
The only modification to $(D)$ is that the second constraint is imposed for all time $t \in [0,T_0]$ instead of just at time $T$. 
All results from the previous sections can be extended to the free-terminal-time case with nearly identical proofs, and the numerical implementation follows in a straightforward manner.
Note in particular the following formulation:
\begin{rem}
\label{rem:freeT}
When $h \equiv 1$ and $H \equiv 0$ with free terminal time, the $(OCP)$ can be interpreted as a minimum time problem, where our task is to find an admissible pair $(x,u)$ such that the trajectory reaches the target set as quickly as possible.
\end{rem}

  \section{Examples}
\label{sec:examples}

In this section, we consider three examples. 
We first consider a pair of examples on the double integrator system which is used to validate the performance of our algorithm.
Next, we consider the application of our algorithm on a system with degenerate linearization: the Dubins Car.

\subsection{Double Integrator}

To benchmark the performance of our approach, we begin by considering the double integrator:
\begin{equation}
\begin{aligned}
\dot{x}_1 & = x_2\\
\dot{x}_2 & = u
\end{aligned}
\end{equation}
where the system state is $x = (x_1, x_2)$ with bounding set $X = \R^2$, input $u$ is restricted in $U = [-1,1]$, and the target set is $X_T = \{(0,0)\}$. 
Even though $X$ is not compact, we may impose the additional constraint $\|x(t)\| \leq N$ for some large $N$ \cite[Section~5.1]{lasserre2008}. 
However, this additional constraint is not enforced in the numerical implementation.

We solve the free terminal time problem with $x_0 = (0.3, 1)$, $h=1$, $H = 0$, and $T_0 = 3$. 
According to Remark \ref{rem:freeT}, this is essentially the problem of finding $u(t) \in U$ such that system state can be driven from $(0.3,1)$ to $(0,0)$ in minimum time. 
For this simple system, the optimal control action and therefore trajectory are analytically computable. 
The optimal solution is compared to the result of our method with $2k=6$, $8$, and $12$ in Figure \ref{fig:DI}. 
To further illustrate we are able to handle more complicated cost functions, we consider the Linear Quadratic Regulator (LQR) problem on double integrator system, where $x_0 = (1,1)$, $X = X_T = \R^2$, $U = [-1,1]$, $h(t,x,u) = x_1^2 + x_2^2 + 20u^2$, $H = 0$, and $T = 5$.
The result of our method with $2k=6$, $8$, and $12$ is compared to a standard finite-horizon LQR solver in Figure \ref{fig:DI_LQR}.

\begin{figure}[!ht]
\begin{minipage}{0.95\columnwidth}
\centering
	\subfloat[Control action \label{fig:DI_control}]{\includegraphics[trim={1cm 6.5cm 2cm 7cm},clip,width=0.49\textwidth]{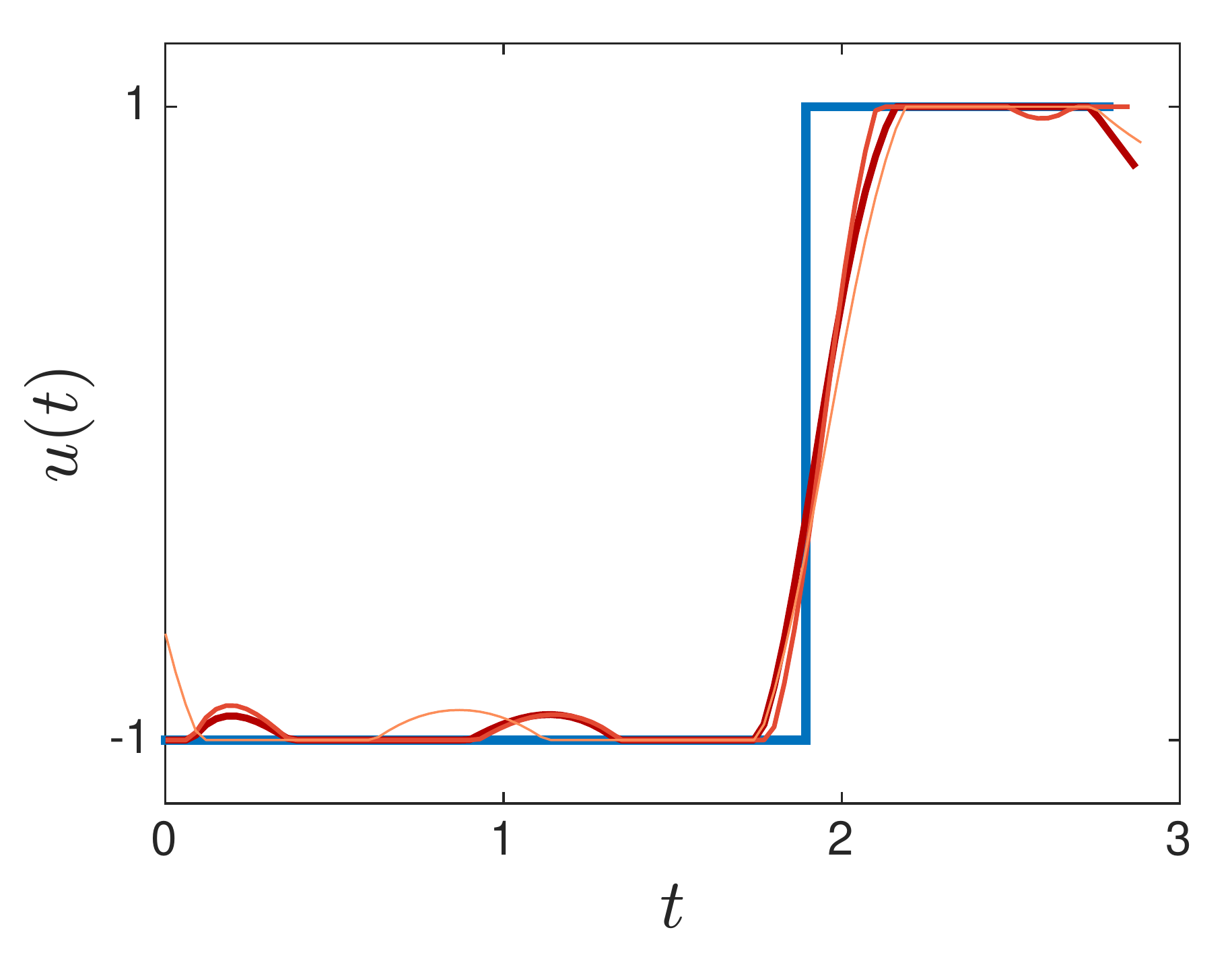}}
    \subfloat[Trajectory\label{fig:DI_traj}]{\includegraphics[trim={1cm 6.5cm 2cm 7cm},clip,width=0.49\textwidth]{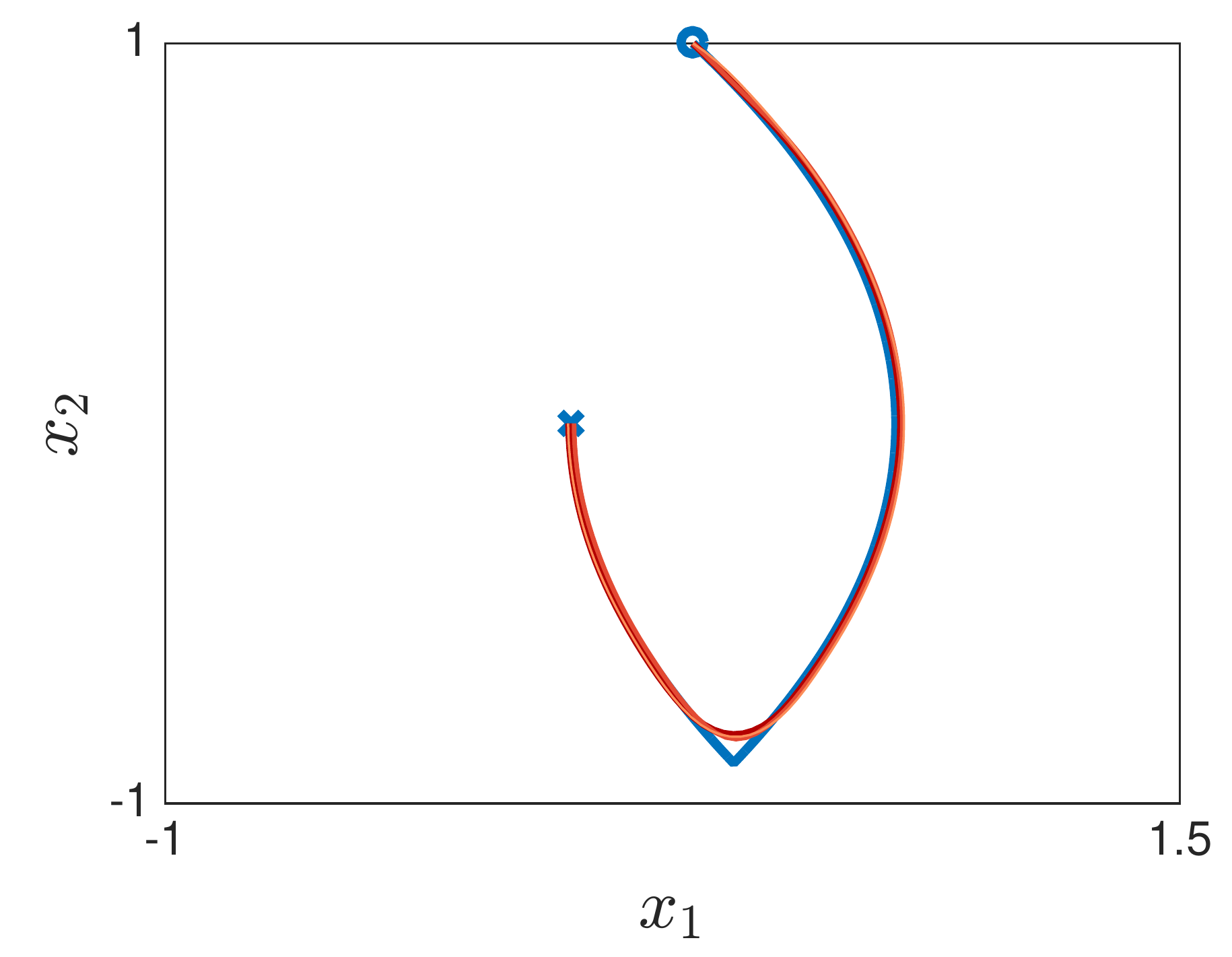}}
	\end{minipage}
\centering
\caption{An illustration of the performance of our algorithm on a free final time version of the double integrator problem. 
The blue circle indicates the given initial state $x_0=(0.3,1)$, and the blue cross shows the target point $(0,0)$.
The blue line is the analytically compute optimal control, while the red lines are control actions generated by our method. 
As the saturation increases the corresponding degree of relaxation increases: $2k=6$, $2k=8$, and $2k=12$. Figure \ref{fig:DI_control} depicts the control action whereas Figure \ref{fig:DI_traj} illustrates the action of the controller when forward simulated through the system.}
\label{fig:DI}
\end{figure}

\begin{figure}[!ht]
\begin{minipage}{0.95\columnwidth}
\centering
	\subfloat[Control action \label{fig:DI_LQR_control}]{\includegraphics[trim={1cm 6.5cm 2cm 7cm},clip,width=0.49\textwidth]{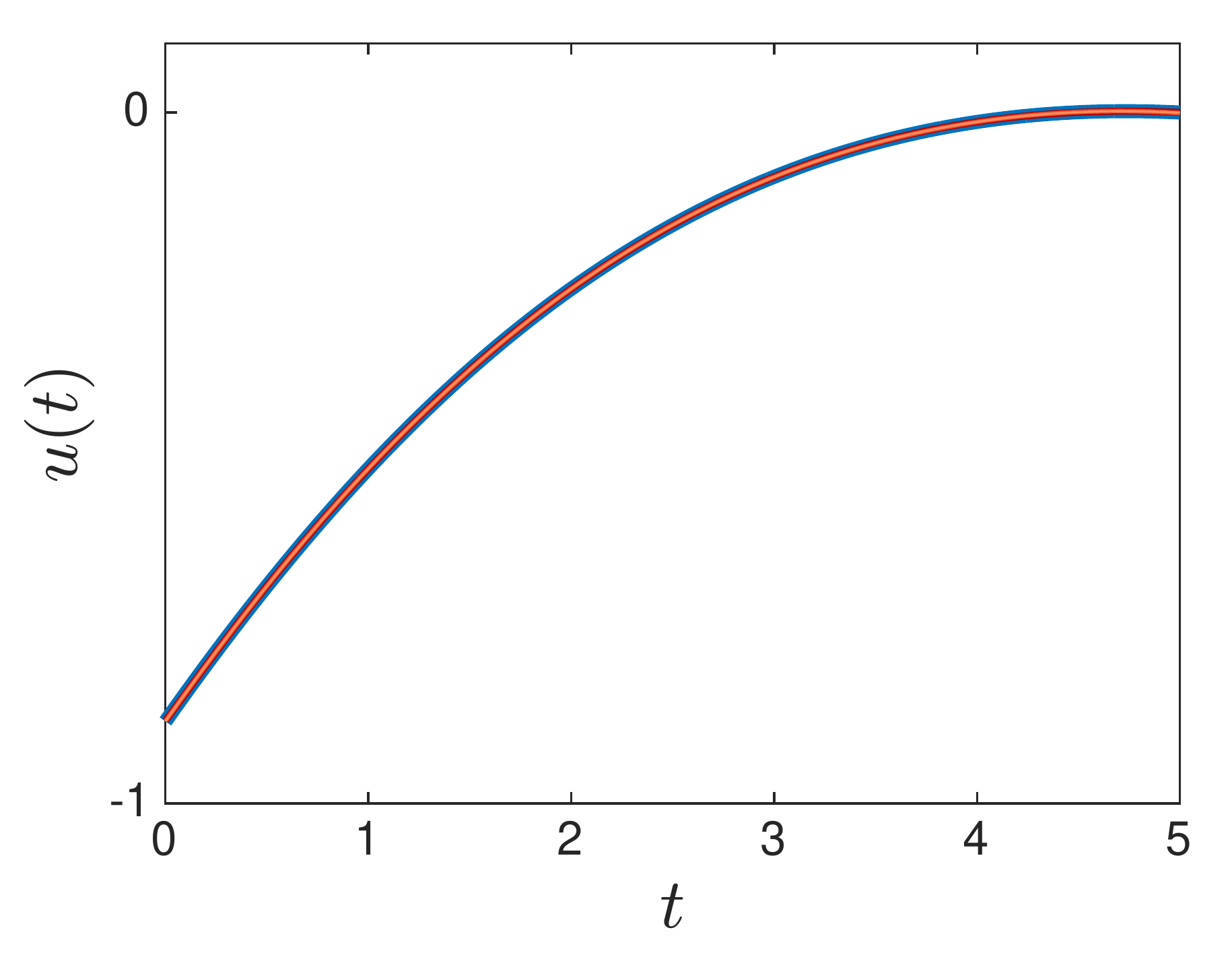}} 
    \subfloat[Trajectory\label{fig:DI_LQR_traj}]{\includegraphics[trim={1cm 6.5cm 2cm 7cm},clip,width=0.49\textwidth]{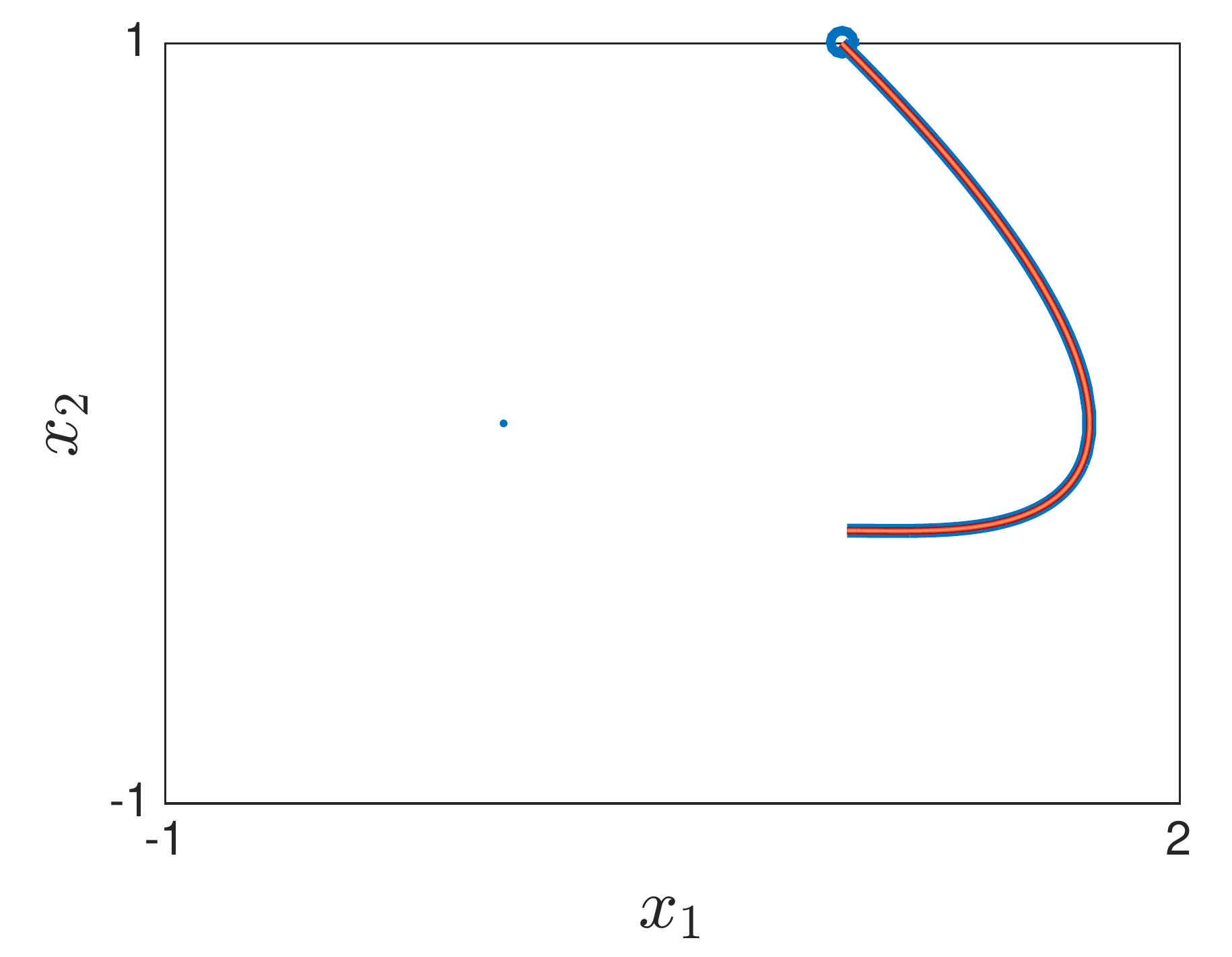}}
	\end{minipage}
\centering
\caption{An illustration of the performance of our algorithm on an LQR version of the double integrator problem. 
The blue circle indicates the given initial state $x_0 = (1,1)$.
The blue line is the analytically computed optimal control, while the red lines are control actions generated by our method. As the saturation increases the corresponding degree of relaxation increases: $2k=6$, $2k=8$, and $2k=12$. Figure \ref{fig:DI_LQR_control} depicts the control action whereas Figure \ref{fig:DI_LQR_traj} illustrates the action of the controller when forward simulated through the system.}
\label{fig:DI_LQR}
\end{figure}

\subsection{Dubins Car Model}

Next, we consider the Dubins Car which is a kinematic car model:
\begin{equation}
\begin{aligned}
\dot{x}_1 &= V \cos (\theta)\\
\dot{x}_2 &= V \sin (\theta)\\
\dot{\theta} &= \omega
\end{aligned}
\end{equation}
where the system state $x = (x_1, x_2, \theta)$ is bounded by $X = [-1,1]\times [-1,1]\times [-\pi,\pi]$, and control input $u = (V, \omega)$ is restricted to $U = [0,1]\times [-3,3]$. 
Although the dynamics are not polynomials, they are approximated by 2nd-order Taylor expansion around $x=(0,0,0)$ in the numerical implementation.
Again, we consider the following LQR problem: the initial condition is $x_0 = (-0.8,0.8,0)$, target set $X_T = X$ is the entire space, $T_0 = 3$, and cost functions are $h = (x_1 - 0.5)^2 + (x_2 + 0.4)^2 + V^2 + (\omega/3)^2$, $H = 0$. 
Note that this problem cannot be solved using a standard LQR solver, because Dubins car has degenerate linearization. 
The result of our method is shown in Figure \ref{fig:Dubins}.

\begin{figure}[!ht]
\begin{minipage}{0.99\columnwidth}
\centering
\subfloat[Control action: $V$\label{fig:dubins_V}]{\includegraphics[trim={1cm 6.5cm 2cm 7cm},clip,width=0.33\textwidth]{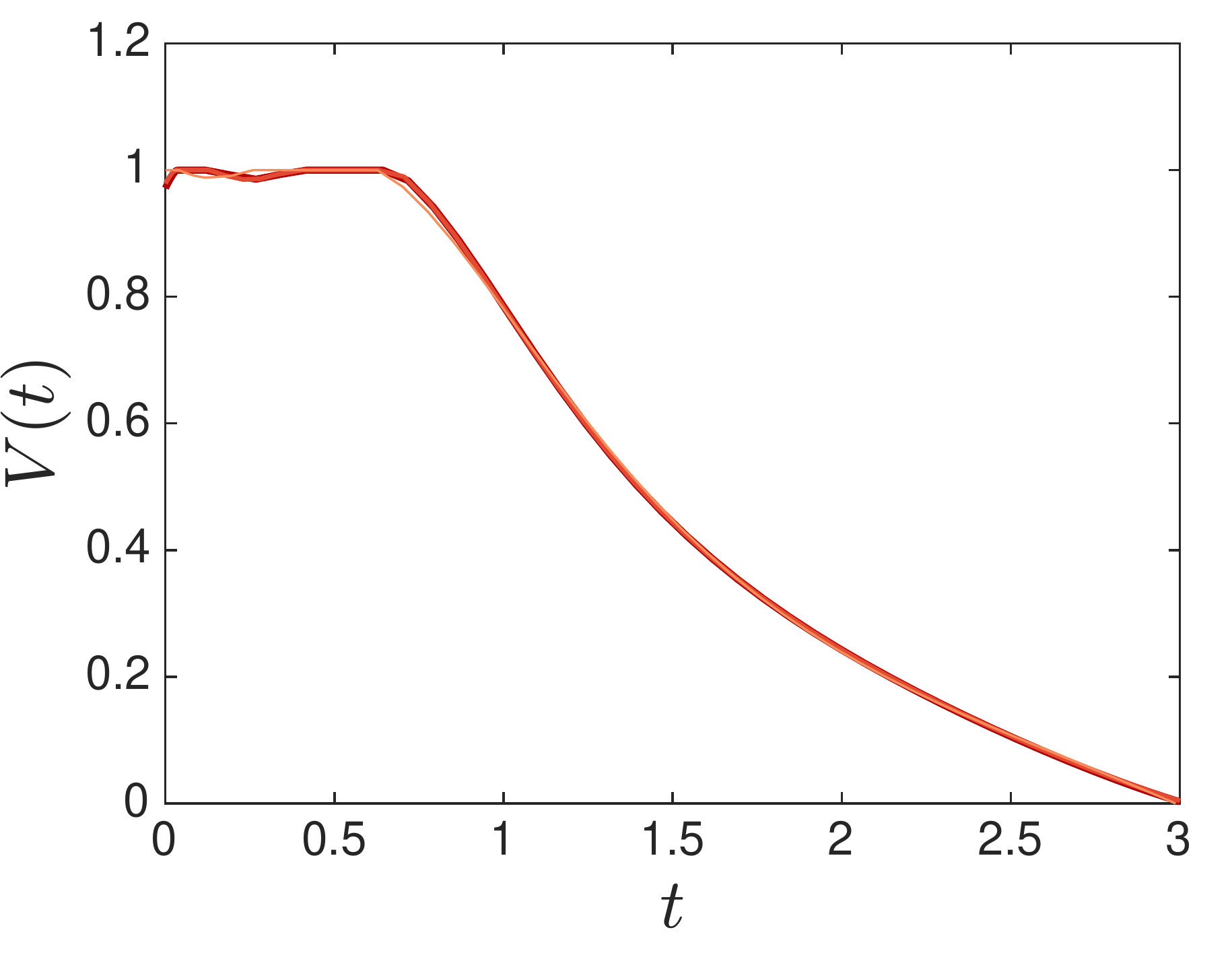}}
\hfil
\subfloat[Control action: $\omega$\label{fig:dubins_omega}]{\includegraphics[trim={1cm 6.5cm 2cm 7cm},clip,width=0.33\textwidth]{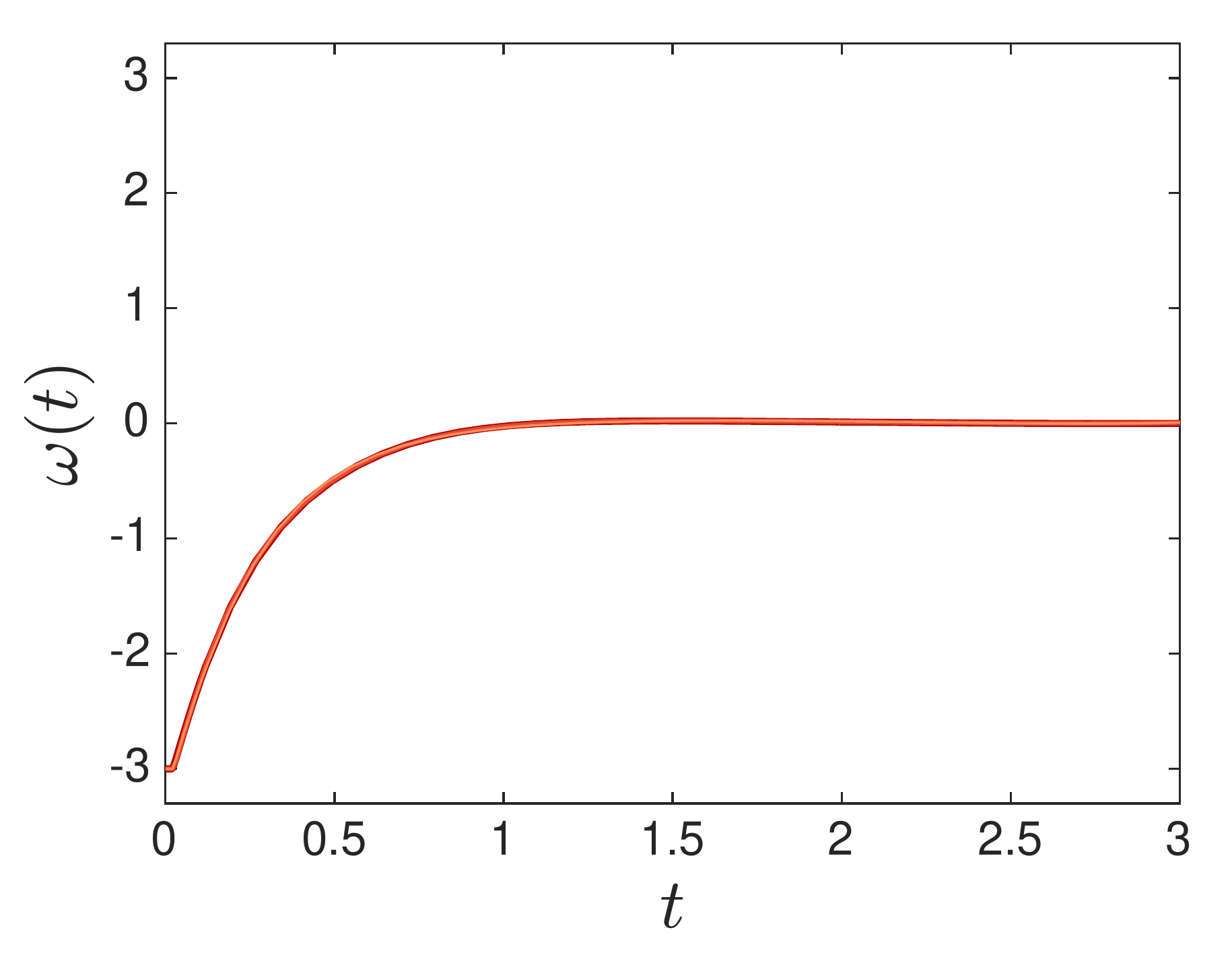}}
\subfloat[Trajectory\label{fig:dubins_traj}]{\includegraphics[trim={1cm 6.5cm 2cm 7cm},clip,width=0.33\textwidth]{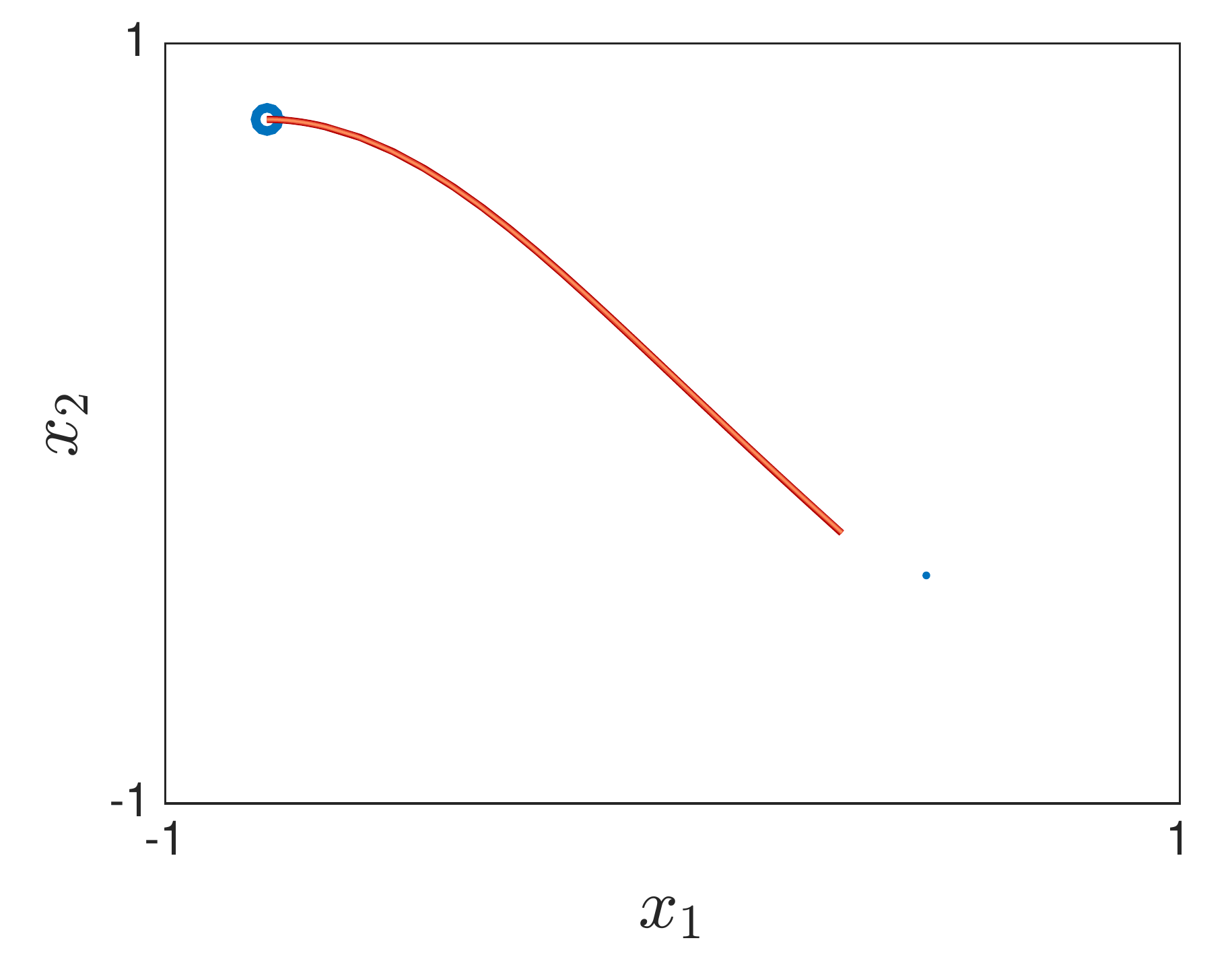}}
\end{minipage}
\caption{ An illustration of the performance of our algorithm on an LQR version of the Dubins Car problem. 
The blue circle indicates the given initial state $x_0$.
The red lines are control actions generated by our method. As the saturation increases the corresponding degree of relaxation increases: $2k=6$, $2k=8$, and $2k=10$. Figures \ref{fig:dubins_V} and \ref{fig:dubins_omega} depict the control action whereas Figure \ref{fig:dubins_traj} illustrates the action of the controller when forward simulated through the system.}
\label{fig:Dubins}
\end{figure}
  \section{Conclusion}

This paper presents an approach for solving the optimal control problem by formulating an infinite-dimensional LP over the space of non-negative measures. 
Finite dimensional SDPs are constructed as relaxations to give both a lower bound of the optimal cost and a polynomial approximation of the optimal control action. 
In contrast to other numerical methods, our approach proposes a convergent convex formulation of the optimal control problem even in the instance of a discontinuous optimal control, which can be solved efficiently on convex optimization program solvers.

\bibliography{root.bib}
\bibliographystyle{unsrt}
\end{document}